\documentclass[12pt,leqno]{amsart}
\usepackage{amsmath,amssymb,amscd,latexsym,amsthm,mathrsfs}
\usepackage{color}
\usepackage[unicode]{hyperref}
\usepackage{hypbmsec,longtable}
\ifx\pdfoutput\undefined\else
\hypersetup{
colorlinks=true,
linkcolor=blue,
citecolor=blue,
urlcolor=blue,
filecolor=blue,
bookmarksnumbered=true,
pdfstartview=FitH,
pdfhighlight=/N
}
\fi
\newtheorem*{maintheorem}{Main theorem}

\newtheorem{lemma}{Lemma}

\newtheorem{proposition}{Proposition}
\theoremstyle{remark}

\newtheorem*{acknowledgements}{Acknowledgements}
\renewcommand{\d}{{\mathrm d}}

\newcommand\m{{\operatorname{m}}}

\begin{document}

\hypersetup{pdfauthor={Mathew Rogers, Wadim Zudilin},%
pdftitle={On the Mahler measure of $1+X+1/X+Y+1/Y$}}

\title{On the Mahler measure of $1+X+1/X+Y+1/Y$}

\author{Mathew Rogers}
\address{Department of Mathematics, University of Illinois, Urbana, IL 61801, USA}
\email{mathewrogers@gmail.com}

\author{Wadim Zudilin}
\address{School of Mathematical and Physical Sciences,
The University of Newcastle, Callaghan, NSW 2308, AUSTRALIA}
\email{wadim.zudilin@newcastle.edu.au}

\thanks{The first author is supported by National Science Foundation award DMS-0803107.
The second author is supported by Australian Research Council grant DP110104419.}

\date{March 29, 2010}


\begin{abstract}
We prove a conjectured formula relating the Mahler measure of the
Laurent polynomial $1+X+X^{-1}+Y+Y^{-1}$ to the $L$-series of a
conductor $15$ elliptic curve.
\end{abstract}

\maketitle
%==================================================

\section{Introduction}
\label{s-intro}

The purpose of this paper is to prove a conjectured identity
relating the Mahler measure of a two-variable Laurent polynomial,
to the $L$-series of a conductor $15$ elliptic curve~\cite{De}:
\begin{equation}\label{deninger identity}
\m\biggl(1+X+\frac{1}{X}+Y+\frac{1}{Y}\biggr)
=\frac{15}{4\pi^2}L(E_{15},2).
\end{equation}
Recall that the (logarithmic) Mahler measure of a polynomial
$P(X_1,\dots,X_n)\in\mathbb C[X_1^{\pm1},\dots,X_n^{\pm1}]$
is the arithmetic mean of $\log|P|$ on the torus
$\mathbb T^n=\{(X_1,\dots,X_n)\in\mathbb C^n:|X_1|=\dots=|X_n|=1\}$,
\begin{equation}
\m(P):=\idotsint_{[0,1]^n}
\log|P(e^{2\pi i\theta_1},\dots,e^{2\pi i\theta_n})|\,\d\theta_1\dotsb\d\theta_n.
\label{Mahler}
\end{equation}
The study of multi-variable Mahler measures originated in the work
of Smyth, who proved relations with Dirichlet $L$-values and
special values of the Riemann zeta function \cite{Sm}. Formula
\eqref{deninger identity} is the first known relation between a
Mahler measure and the $L$-series of an elliptic curve. The
original formulation is due to Deninger, who proved that the
identity follows, up to an unknown rational factor, from the
Beilinson conjectures~\cite{De}. Boyd subsequently calculated the
rational factors, and also found~\cite{Bo1} that similar
identities were numerically true for the polynomial family
$k+X+X^{-1}+Y+Y^{-1}$ whenever $k\in\mathbb{Z}$. Bertin \cite{BE}
and Rodriguez-Villegas \cite{RV} have also investigated Mahler
measures of elliptic curves.

The authors recently proved Boyd's conjectures for
non-CM elliptic curves of conductors $20$ and $24$ \cite{RZ}.  The
basic idea was to manipulate the cusp forms associated with the
elliptic curves, in order to obtain elementary integrals for the
$L$-values. In the conductor $20$ case, it was shown that
\begin{equation}\label{conductor 20}
L(E_{20},2)=-\frac{\pi}{20}\int_{0}^{1}\frac{(1-6t)\log(1+4t)}{\sqrt{t(1-t)(1+4t^2)}}\,\d t.
\end{equation}
The integrals were then related to Mahler measures through an
intricate analysis of hypergeometric functions.  Formula~\eqref{conductor 20}
can be reduced to a Mahler measure by setting
$k=4$ in an identity valid for $k\in[2,8]$:
\begin{equation*}
\frac{1}{2\pi}\int_{0}^{1}\frac{(2-k+3k t)\log(1+kt)}{\sqrt{t(1-t)(4+(4-k)k t+k^2 t^2)}}\,\d t
=\m\bigl((1+X)(1+Y)(X+Y)-kXY\bigr).
\end{equation*}
As might be expected, these sorts of integrals are difficult to
analyze. It required a great deal of trial and error to introduce
the parameter~$k$ into~\eqref{conductor 20}. Even with the correct
definition, it was very difficult to relate the integral to a
Mahler measure.

We will use elementary techniques to prove formula~\eqref{deninger
identity}.  Our method relies upon integrating Ramanujan's modular
equations, and is applicable to many different elliptic curves.
Prior to our work, the only method for attacking Boyd's
conjectures centered around Beilinson's theorem. Brunault and
Mellit used Beilinson's theorem to prove Boyd's conjectures for
conductor $11$ and $14$ elliptic curves \cite{Br}, \cite{Me}. We
expect to present elementary proofs of their results in a future
paper.  It is important to mention the fact that Zagier and
Kontsevich predicted the existence of formulas like
\eqref{conductor 20}, as a consequence of Beilinson's theorem
\cite[\S\,3.4]{KZ}.  While our current method is independent of
such $K$-theoretic considerations, it seems that the two
approaches yield overlapping results. We believe that this current
work, and our previous paper \cite{RZ}, are the first instances
where elementary formulas such as \eqref{conductor 20} have been
explicitly stated.

We will introduce two new ideas in this paper.  The first is that
it is possible to express \emph{many} different $L$-values in
terms of a single function $H(x)$, the function introduced
in~\cite{RZ}. If we consider the eta function with respect to~$q$,
\begin{equation*}
\eta(q):=q^{1/24}\prod_{n=1}^{\infty}(1-q^{n}),
\end{equation*}
and define the signature~$3$ theta functions~\cite[Chap.~33]{Be5},~\cite{Borwein}
\begin{equation}\label{3theta}
\begin{gathered}
a(q):=\sum_{m,n\in\mathbb Z}q^{m^2+mn+n^2}, \quad
b(q):=\sum_{m,n\in\mathbb Z}e^{2\pi
i(m-n)/3}q^{m^2+mn+n^2}=\frac{\eta^3(q)}{\eta(q^3)},
\\
\quad\text{and}\quad c(q):=\sum_{m,n\in\mathbb
Z}q^{(m+1/3)^2+(m+1/3)(n+1/3)+(n+1/3)^2}=3\frac{\eta^3(q^3)}{\eta(q)},
\end{gathered}
\end{equation}
then $H(x)$ is given by
\begin{equation}\label{H(x) definition}
H(x):=\int_{0}^{1}\frac{\eta^3(q^3)}{\eta(q)}\,\frac{\eta^3(q^x)}{\eta(q^{3x})}\,\log q\,\frac{\d q}{q}
=\frac{1}{3}\int_{0}^{1}b(q^x) c(q)\log q\,\frac{\d q}{q}.
\end{equation}
If we recall that the conductor $27$ elliptic curve is associated
to $\eta^2(q^3) \eta^2(q^9)$ \cite{Ono}, then it is simple to see
that
\begin{equation*}
-9L(E_{27},2)=H(1).
\end{equation*}
It is much less obvious that the $L$-series of a conductor $15$
elliptic curve also reduces to values of $H(x)$.  We will use a
telescoping modular equation to prove that
\begin{equation}\label{F(3,5) in terms of H intro}
-45L(E_{15},2)=\frac{1}{5}H\biggl(\frac{1}{15}\biggr)
+5H\biggl(\frac{5}{3}\biggr)+\frac{4\pi^2}{3}\log3.
\end{equation}
We have discovered that at least $9$ different $L$-values can be
related to $H(x)$; those formulas are presented in the next
section.  The definition of $H(x)$ was initially guessed after
examining the complex-multiplication, conductor $27$, example.
There are several additional functions which possess properties
analogous to $H(x)$. Those functions can be used to prove many
additional relations between Mahler measures and $L$-series of
elliptic curves, and will be examined in forthcoming papers.

The second idea we will require, is that certain linear
combinations of $H(x)$ can be reduced to elementary integrals.
These identities go well beyond the scope of our previous analysis
in~\cite{RZ}. For any $x\in\mathbb{Q}\cap(0,\infty)$, there exists
a polynomial relation between $u$ and $v$, such that
\begin{align}
\label{H main theorem in intro}
&
xH\biggl(\frac{x}{3}\biggr)+\frac{1}{x}H\biggl(\frac{1}{3x}\biggr)+2H\biggl(\frac{1}{3}\biggr)
\\ &\quad
=4\pi\int_{v\in[0,1]}\log\biggl(\frac{1-R+v}{u}\biggr)
\d\arctan\biggl(\frac{\sqrt{3}(1+R)}{1-R-2v}\biggr),
\nonumber
\end{align}
where $R^3-3v R-v^3=1-u^3$. The right-hand side of~\eqref{H main
theorem in intro} is essentially a function of a polynomial. While
the value of~$x$ dictates the choice of polynomial, it should be
possible to evaluate the integral for any sufficiently simple
algebraic relation between $u$ and~$v$.  When $x=2$ the
corresponding relation is $u+v-1=0$, and when $x=5$ the relation
is given by $(u+v-1)^2-9uv=0$. We will use formulas \eqref{F(3,5)
in terms of H intro} and~\eqref{H main theorem in intro} to prove
the conductor $15$ conjecture.

\section{Telescoping modular equations and numerical conjectures}
\label{sec2}
In this section we will prove formulas relating $L$-values to
$H(x)$ defined in~\eqref{H(x) definition}.
We will also present some unproven formulas, which hold to
high numerical precision. A number of similar formulas were proved
in~\cite{RZ}. The basic idea in the previous paper, was to
decompose cusp forms into signature~$3$ theta functions. For
instance, integrating the formula
\begin{equation}\label{conductor 36 modular equation}
3\eta^4(q^6)=b(q^4)c(q^3)-b(q)c(q^{12}),
\end{equation}
leads to a linear relation between $L(E_{36},2)$, $H(4/3)$ and
$H(1/12)$ \cite{RZ}. Recall that the $L$-series of a conductor
$36$ elliptic curve equals the Mellin transform of $\eta^4(q^6)$
\cite{Ono}.  We can state this result in the form
$L(E_{36},2)=F(1,1)$, where the quadruple sum
\begin{equation}
\label{F(b,c)}
\begin{split}
F(B,C)&:=(B+1)^2(C+1)^2
\\ &\qquad
\times\sum_{\substack{n_i=-\infty\\i=1,2,3,4}}^{\infty}
\frac{(-1)^{n_1+n_2+n_3+n_4}}{\bigl((6n_1+1)^2+B(6n_2+1)^2+C(6n_3+1)^2+BC(6n_4+1)^2\bigr)^2}
\end{split}
\end{equation}
can be identified as the $L$-series $L(f,2)$ of the cusp form
$$
f(q)=\eta(q^A)\eta(q^{AB})\eta(q^{AC})\eta(q^{ABC}),
\quad A=\frac{24}{(B+1)(C+1)},
$$
whenever $A$ is an integer.

Many of the new formulas in this section, follow from integrating
`telescoping' modular equations.
The key is to search for identities which are similar to
\eqref{conductor 36 modular equation}, but which involve
additional terms.  When the telescoping terms are integrated,
their modularity properties can be used to evaluate them
explicitly. The following formula sets the grounds of this idea.

\begin{lemma}
\label{telescope}
For $r>0$ and $j>0$,
\begin{equation}\label{main telescoping formula}
\int_{0}^{1}\bigl(r^2c(q^{r})c(q^{rj})-c(q)c(q^j)\bigr)\log q\,\frac{\d q}{q}
=\frac{4\pi^2}{3j}\log r.
\end{equation}
\end{lemma}

We will illustrate the utility of this approach with an example.
Consider the following modular equation:
\begin{equation}\label{n=2 telescoping modular equation}
0=a(q)a(q^2)-b(q)b(q^2)-c(q)c(q^2).
\end{equation}
Formula \eqref{n=2 telescoping modular equation} is a $q$-version
of the second degree modular equation in Ramanujan's theory of
signature $3$ \cite[Theorem~2.6]{BoG}. Eliminating $a(q)$ and $a(q^2)$
with a standard relation, $a(q)=b(q)+3c(q^3)$, brings the equation
to
\begin{equation*}
3b(q)c(q^6)+3b(q^2)c(q^3)=-9c(q^3)c(q^6)+c(q)c(q^2).
\end{equation*}
Now multiply both sides by $(2\log q)/q$, and integrate for
$q\in(0,1)$. The left-hand side immediately reduces to values of
$H(x)$. Applying Lemma~\ref{telescope} with $r=3$ and $j=2$, yields
\begin{equation}\label{H(1/6) and H(2/3) relation}
\frac{1}{2}H\biggl(\frac{1}{6}\biggr)+2H\biggl(\frac{2}{3}\biggr)
=-\frac{4\pi^2}{3}\log3.
\end{equation}
Thus we have obtained a $\mathbb Q$-linear dependency between $H(1/6)$,
$H(2/3)$, and $\pi^2\log3$.  There are at least three variants
of \eqref{n=2 telescoping modular equation} which lead to formulas
for $L$-functions of elliptic curves.

\begin{proposition}\label{Lemma on special values of H}
The following relations are either proved or hold numerically:
{\allowdisplaybreaks
\begin{align}
\frac{4\pi^2}{3}\log3
&=-\frac{1}{x}H\biggl(\frac{1}{3x}\biggr)+3xH\biggl(\frac{x}{3}\biggr)
-3xH(x)+\frac{1}{x}H\biggl(\frac{1}{9x}\biggr), \label{H(x)
functional relation}
\\
\pi\sqrt{3}L(\chi_{-3},2)
&=-H\biggl(\frac{1}{3}\biggr),
\label{evaluation of H(1/3)}
\\
12L(E_{14},2)=12F(2,7)
&=-\frac{1}{14^2}H\biggl(\frac{1}{42}\biggr)-H\biggl(\frac{14}{3}\biggr)
\label{F(2,7) in terms of H}
\\ &\qquad
+\frac{1}{7^2}H\biggl(\frac{2}{21}\biggr)+\frac{1}{2^2}H\biggl(\frac{7}{6}\biggr),
\notag\\
9L(E_{15},2)=9F(3,5)
&=-\frac{1}{5^2}H\biggl(\frac{1}{15}\biggr)-H\biggl(\frac{5}{3}\biggr)-\frac{4\pi^2}{15}\log3,
\label{evaluation of F(3,5) in terms of H}
\\
24L(E_{20},2)=24F(1,5)
&\stackrel{?}{=}-\frac{2}{5^2}H\biggl(\frac{1}{15}\biggr)+2H\biggl(\frac{5}{3}\biggr)+\frac{4}{5^2}H\biggl(\frac{4}{15}\biggr)
\label{F(1,5) in terms of H}
\\ &\qquad
-4H\biggl(\frac{20}{3}\biggr)+\frac{3}{5}H\biggl(\frac{1}{3}\biggr),
\notag\\
9L(E_{24},2)=9F(2,3)
&=-\frac{1}{8^2}H\biggl(\frac{1}{24}\biggr)-H\biggl(\frac{8}{3}\biggr)-\frac{\pi^2}{6}\log3,
\label{F(2,3) in terms of H}
\\
9L(E_{27},2)=9F(1,3)
&=-H(1),
\label{F(1,3) in terms of H}
\\
9L(E_{36},2)=9F(1,1)
&=-H\biggl(\frac{4}{3}\biggr)+\frac{1}{4^2}H\biggl(\frac{1}{12}\biggr)
\label{F(1,1) in terms of H}
\\
&\stackrel{?}{=}2H\biggl(\frac{1}{3}\biggr)-2H\biggl(\frac{4}{3}\biggr),
\label{F(1,1) in terms of H conj}
\\
6L(E_{33},2)+4L(E_{11},2)&=12F(1,11)+\frac{9}{2}F(3,11)\label{F(1,11) and F(3,11) in terms of H}\\
&=-\frac{1}{11^2}H\biggl(\frac{1}{33}\biggr)-H\biggl(\frac{11}{3}\biggr)-\frac{4\pi^2}{33}\log3,\notag
\\
\frac{27}{16}F(3,7)
&=\frac{8}{7}H(1)-H(7)-\frac{1}{49}H\biggl(\frac{1}{7}\biggr),
\label{F(3,7) in terms of H}
\\
\frac{27}{49}F(6,7)
&=\frac{1}{49}H\biggl(\frac{2}{7}\biggr)+H(14)-\frac{8}{7}H(2),
\label{F(6,7) in terms of H}
\\
\frac{27}{25}F\biggl(\frac{3}{2},7\biggr)
&=\frac{2}{7}H\biggl(\frac{1}{2}\biggr)-\frac{1}{4}H\biggl(\frac{7}{2}\biggr)-\frac{1}{14^2}H\biggl(\frac{1}{14}\biggr),
\label{F(3/2,7) in terms of H}
\end{align}}%
where $\chi_{-3}$ in~\eqref{evaluation of H(1/3)} denotes the non-principal character $\mod3$.
\end{proposition}

\begin{proof}
We will begin by proving formulas \eqref{evaluation of F(3,5) in
terms of H}, \eqref{F(2,3) in terms of H} and \eqref{F(1,11) and
F(3,11) in terms of H}.  The proofs follow in exactly the same
manner as the proof of \eqref{H(1/6) and H(2/3) relation}, using
the telescoping identity of Lemma~\ref{telescope}. The
relevant modular equations are due to Ramanujan:
\begin{align*}
a(q)a(q^5)-b(q)b(q^5)-c(q)c(q^5)
&=9\eta(q)\eta(q^3)\eta(q^5)\eta(q^{15}),
\\
a(q)a(q^8)-b(q)b(q^8)-c(q)c(q^8)
&=9\eta(q^2)\eta(q^4)\eta(q^6)\eta(q^{12}),
\\
a(q)a(q^{11})-b(q)b(q^{11})-c(q)c(q^{11})
&=9\eta^2(q)\eta^2(q^{11})+27\eta^2(q^3)\eta^2(q^{33})
\\ &\qquad
+18\eta(q)\eta(q^{3})\eta(q^{11})\eta(q^{33}).
\end{align*}
The first modular equation is equivalent to \cite[pg.~125, Entry 7.20]{Be5}.
The equivalence follows from using product expansions
$b(q)=\eta^3(q)/\eta(q^3)$,
$c(q)=3\eta^3(q^3)/\eta(q)$, and the cubic relation
$a(q)=(b^3(q)+c^3(q))^{1/3}$. The second result follows from
\cite[pg.~129, Entry 7.40]{Be5}, and the third follows from
\cite[pg.~127, Entry 7.30]{Be5}.  The identification of
$L(E_{33},2)$ in terms of $F(1,11)$ and $F(3,11)$, follows from
integrating the associated cusp form \cite{Pat}.

In order to prove \eqref{evaluation of H(1/3)}, we can use
\cite[pg.~217, Entry~14.6]{GK}, to obtain
\begin{equation*}
\frac{1}{3}b(q)c(q^3)=\sum_{n,k=1}^{\infty}k\chi_{-3}(n k) q^{nk}.
\end{equation*}
Multiplying by $(\log q)/q$ and integrating for $q\in(0,1)$ on either side,
we have
\begin{equation*}
\begin{split}
\frac{1}{3^2}H\biggl(\frac{1}{3}\biggr)
=-L(\chi_{-3},1)L(\chi_{-3},2)
=-\frac{\pi}{3\sqrt{3}}L(\chi_{-3},2).
\end{split}
\end{equation*}

The proof of \eqref{H(x) functional relation} follows from
integrating the following identity:
\begin{equation*}
\bigl(b(q^{1/9})-b(q^{1/3})\bigr)c(q^x)+3\bigl(b(q^{x/3})-b(q^x)\bigr)c(q)
=9c(q^{3x})c(q)-c(q^x)c(q^{1/3}),
\end{equation*}
and then applying \eqref{main telescoping formula} to the right-hand
side.  This identity can be verified by eliminating $b(q)$ with
$$
b(q^{1/3})-b(q)=3c(q^3)-c(q).
$$
The simpler identity between $b$ and $c$ is a consequence
of \cite[pg.~93, Entry 2.8]{Be5} and \cite[pg.~94, Entry 2.9]{Be5}.

Finally, identities \eqref{F(1,3) in terms of H},
\eqref{F(1,1) in terms of H}, \eqref{F(3,7) in terms of H}
\eqref{F(6,7) in terms of H}, and \eqref{F(3/2,7) in terms of H}
are examined in~\cite[Lemma~1]{RZ}.
\end{proof}

\begin{proof}[Proof of Lemma~\textup{\ref{telescope}}]
Let us denote the left-hand side of~\eqref{main telescoping
formula} by~$I_{j,r}$.  Notice that
\begin{equation*}
I_{j,r}=\lim_{\delta\to1}\biggl(\int_{0}^{\delta}r^2c(q^{r})c(q^{r j})\log q\,\frac{\d q}{q}
-\int_{0}^{\delta}c(q)c(q^j)\log q\,\frac{\d q}{q}\biggr).
\end{equation*}
The rearrangement is justified because $c(q)=O(q^{1/3})$ as $q\to 0^+$.
Performing a $u$-substitution brings the difference to
\begin{equation*}
I_{j,r}=\lim_{\delta\to1}\int_{\delta}^{\delta^r}c(q)c(q^{j})\log q\,\frac{\d q}{q}.
\end{equation*}
It is known that $b(q)$ and $c(q)$ are linked by the modularity relation
\begin{equation*}
c(q)=-\frac{2\pi}{\sqrt{3}\log q} b\bigl(e^{(4\pi^2)/(3\log q)}\bigr).
\end{equation*}
When $q\to1^{-}$ it is easy to see that
$e^{(4\pi^2)/(3\log q)}\to 0$.  Since $b(0)=1$, we estimate
\begin{equation*}
c(q)=-\frac{2\pi}{\sqrt{3}\log q}\bigl(1+O(1-q)\bigr)
\quad\text{as $q\to1^-$}.
\end{equation*}
Substituting for $c(q)$ and $c(q^j)$ reduces the integral to
\begin{equation*}
\begin{split}
I_{j,r}
&=\lim_{\delta\to1}\frac{4\pi^2}{3j}\int_{\delta}^{\delta^r}\frac{1}{\log q}\bigl(1+O(1-q)\bigr)\frac{\d q}{q}
\\
&=\frac{4\pi^2}{3j}\log r+\lim_{\delta\to 1}\int_{\delta}^{\delta^r}O\biggl(\frac{1-q}{q\log q}\biggr)\d q
\\
&=\frac{4\pi^2}{3j}\log r.
\end{split}
\end{equation*}
The error term is the tail of the convergent integral
$$
\int_{1/2}^{1}\frac{1-q}{q\log q}\,\d q,
$$
and vanishes as $\delta\to1$. This concludes the proof of~\eqref{main
telescoping formula}.
\end{proof}

To finish this discussion, we will emphasize the fact that formula
\eqref{evaluation of F(3,5) in terms of H} is the real prize of
Proposition~\ref{Lemma on special values of H}.   That formula
provides one of the keys to solving the conductor $15$ conjecture.
Formula \eqref{F(2,7) in terms of H} also seems promising, however
we basically ignored the identity in our analysis, because Mellit
has already proved Boyd's conjectures for conductor $14$
curves~\cite{Me}. The proof of \eqref{F(2,7) in terms of H} is
also quite difficult, and as a result we have chosen not to
present it here. This brings us to the numerical identities.  We
were disappointed that we could not isolate the conductor $11$
case, since that $L$-value appears frequently in Boyd's tables.
Equation \eqref{F(1,5) in terms of H} is the most interesting
formula that we were able to conjecture, since it involves the
conductor $20$ elliptic curve. We discovered \eqref{F(1,5) in
terms of H} with the PSLQ algorithm, and were subsequently unable
to prove it. The problem is that there is no obvious way to relate
integrals of $\eta^2(q^2) \eta^2(q^{10})$ to values of $H(x)$.  We
performed an extensive, but ultimately futile, search of
Ramanujan's formulas \cite{Be5} and Somos's identities
\cite{so2}.  Eventually we chose to bypass the problem, because
Boyd's conductor $20$ conjectures are already proved~\cite{RZ}.

It seems likely that a more extensive search will turn up many
additional results.  Our primary goal was to find formulas for the
lattice sums $F(B,C)$ defined in~\eqref{F(b,c)}.  Since the vast
majority of elliptic curves have a value of $L(E,2)$ which is
(presumably) linearly independent from $F(B,C)$ over~$\mathbb Q$,
our search would not have addressed those cases. There are also
many additional $\mathbb Q$-linear dependencies between values of
$H(x)$. The majority of these formulas are probably insignificant,
and we expect that most of them can be proved with a telescoping
recipe. For instance, we calculated
\begin{align*}
5H(1)&\stackrel{?}{=}4H(4)+\frac{1}{4}H\biggl(\frac{1}{4}\biggr),
\\
3H(2)&\stackrel{?}{=}4H\biggl(\frac{2}{3}\biggr)+\frac{1}{4}H\biggl(\frac{1}{18}\biggr).
\end{align*}
The only continuous identity that we discovered was \eqref{H(x)
functional relation}.  Perhaps it is noteworthy that this
functional relation can be combined with our other results, to
either prove or conjecture explicit formulas for $H(1)$, $H(1/3)$,
$H(2/3)$, $H(1/6)$, and $H(1/9)$.

\section{A new integral for $H(x)$}
\label{sec3}

One of our main theorems in~\cite{RZ}, is that it is always
possible to express $H(x)$ as an integral of elementary functions.
Suppose that $x>0$, and assume that $\beta$ has degree $x$ over
$\alpha$ in the theory of signature~$3$. Then it was proved that
\begin{equation}\label{H reduction integral 1}
x H\biggl(\frac{x}{3}\biggr)
=\frac{2\pi}{\sqrt{3}}
\int_{0}^{1}\frac{(1-\alpha)^{1/3}\bigl(1-(1-\alpha)^{1/3}\bigr)}{\alpha(1-\alpha)}
\,\log\frac{1-(1-\beta)^{1/3}}{\beta^{1/3}}\,\d\alpha.
\end{equation}
We say that $\beta$ has degree $x$ over $\alpha$ in signature~$3$,
if $\alpha$ and $\beta$ can be parameterized by
\begin{equation*}
\alpha=\frac{c^3(q)}{a^3(q)}, \quad
\beta=\frac{c^3(q^x)}{a^3(q^x)},
\end{equation*}
where $a(q)$ and $c(q)$ are the signature~$3$ theta functions~\eqref{3theta}.
The existence of signature~$3$ modular equations is a consequence of
the classical theory of modular forms.  If $q=e^{2\pi i \tau}$,
then $\alpha$ and $\beta$ are algebraic functions of $j(\tau)$ and
$j(x \tau)$, and therefore satisfy an algebraic relation whenever
$x\in\mathbb{Q}\cap(0,\infty)$.

\begin{proposition}\label{H Theorem}
Suppose that $x>0$.  Then we have
\begin{align}
\label{H main theorem}
&
xH\biggl(\frac{x}{3}\biggr)+\frac{1}{x}H\biggl(\frac{1}{3x}\biggr)+2H\biggl(\frac{1}{3}\biggr)
\\ &\quad
=4\pi\int_{v\in[0,1]}\log\biggl(\frac{1-R+v}{u}\biggr)
\d\arctan\biggl(\frac{\sqrt{3}(1+R)}{1-R-2v}\biggr),
\nonumber
\end{align}
where $R^3-3v R-v^3=1-u^3$.  There is another algebraic relation
between $u$ and $v$ whenever $x\in \mathbb{Q}\cap(0,\infty)$. This
relation is induced by the parameterizations
$u=(\alpha\beta)^{1/3}$ and
$v=\bigl((1-\alpha)(1-\beta)\bigr)^{1/3}$, where $\beta$ has
degree~$x$ over $\alpha$ in signature~$3$. The following table
lists the relation for the first few cases:
\begin{equation*}
    \begin{tabular}{|c|c|c|c|p{6 in}|}%{|p{1 in}|p{2.5 in}|p{2.25 in}|}
        \hline
        $x$ &  $\text{algebraic relations between $u$ and $v$}$ \\
        \hline
        $2$ & $u+v-1$\\
        $5$ & $(u+v-1)^2-9u v$ \\
        $8$ & $(u+v-1)^4+9u v(4u+4v+5)(u+v-1)-162 u^2 v^2$\\
        $11$ & $u+v+6\sqrt{u v}+3\sqrt{3}\sqrt[4]{u v}(\sqrt{u}+\sqrt{v})-1$\\
        \hline
    \end{tabular}
\end{equation*}
\end{proposition}

\begin{proof}
First notice \eqref{H reduction integral 1} simplifies to
\begin{equation*}
xH\biggl(\frac{x}{3}\biggr)
=-4\pi\int_{0}^{1}\log\frac{1-(1-\beta)^{1/3}}{\beta^{1/3}}\,\d\arctan\frac{1+2(1-\alpha)^{1/3}}{\sqrt{3}}.
\end{equation*}
If we set $x=1$, then $\alpha=\beta$, and we obtain an integral
for $H(1/3)$. Add the two formulas together, and notice that
\begin{equation*}
\log\frac{(1-(1-\alpha)^{1/3})(1-(1-\beta)^{1/3})}{(\alpha\beta)^{1/3}}
=\log\biggl(\frac{1-R+v}{u}\biggr),
\end{equation*}
where $u=(\alpha\beta)^{1/3}$, $v=\bigl((1-\alpha)(1-\beta)\bigr)^{1/3}$, and
$R=(1-\alpha)^{1/3}+(1-\beta)^{1/3}$.  Notice that $R^3-3vR-v^3=1-u^3$.
The identity becomes
\begin{equation*}
x H\biggl(\frac{x}{3}\biggr)+H\biggl(\frac{1}{3}\biggr)
=-4\pi\int_{\alpha\in[0,1]}\log\biggl(\frac{1-R+v}{u}\biggr)
\d\arctan\frac{1+2(1-\alpha)^{1/3}}{\sqrt{3}}.
\end{equation*}
The transformation $x\mapsto 1/x$ swaps $\alpha$ and
$\beta$ in the integral.  The limits of integration are unchanged,
because $\alpha=0$ when $\beta=0$, and $\alpha=1$ when $\beta=1$.
The function inside the logarithm is unchanged, because it is
symmetric in $\alpha$ and~$\beta$. The integral becomes
\begin{equation*}
\frac{1}{x}H\biggl(\frac{1}{3x}\biggr)+H\biggl(\frac{1}{3}\biggr)
=-4\pi\int_{\alpha\in[0,1]}\log\biggl(\frac{1-R+v}{u}\biggr)
\d\arctan\frac{1+2(1-\beta)^{1/3}}{\sqrt{3}}.
\end{equation*}
Now add the formulas for $H(x/3)$ and $H(1/(3x))$,
and use the addition formula for $\arctan z$ to complete the proof of~\eqref{H main theorem}.
Notice that $v\in[1,0]$ when $\alpha\in[0,1]$, and $v$~is
monotone, therefore we can express the limits of integration in
terms of~$v$.

The specific algebraic relations between $u$ and $v$ are
equivalent to modular equations in Ramanujan's theory of signature~$3$.
The second degree modular equation \cite[pg.~120, Theorem~7.1]{Be5} shows that
\begin{equation*}
(\alpha\beta)^{1/3}+\bigl((1-\alpha)(1-\beta)\bigr)^{1/3}=1,
\end{equation*}
which is equivalent to $u+v-1=0$.  The fifth degree modular
equation \cite[pg.~124, Theorem~7.6]{Be5} shows that
\begin{equation*}
(\alpha\beta)^{1/3}+\bigl((1-\alpha)(1-\beta)\bigr)^{1/3}+3\bigl(\alpha\beta(1-\alpha)(1-\beta)\bigr)^{1/6}=1;
\end{equation*}
this is equivalent to $(u+v-1)^2-9uv=0$. Finally, cases $x=8$ and
$x=11$ follow from \cite[pg.~132, Theorem~7.11]{Be5} and
\cite[pg.~126, Theorem~7.8]{Be5}, respectively.
\end{proof}

\section{Simplification for $x=2$}
\label{sec4}

Before we attack the conductor $15$ conjecture, we will briefly
examine the much easier case when $x=2$.  Notice that we have
already evaluated the left-hand side of~\eqref{H main theorem} in
formula \eqref{H(1/6) and H(2/3) relation}.  In fact, the
following analysis will recover the correct identity.  When $x=2$
the relations between $u$, $v$, and $R$ are given by
\begin{equation*}
R^3-3v R-v^3=1-u^3, \quad u+v=1.
\end{equation*}
Therefore we have a genus~0 curve relating $v$ and~$R$:
\begin{equation*}
R^3-3v R-v^3=1-(1-v)^3.
\end{equation*}
\texttt{Maple} produces the following rational parameterizations:
\begin{equation*}
R=\frac{(t-1)(2t^2-t+2)}{t^3+2}, \quad v=\frac{(t-1)^3}{t^3+2}.
\end{equation*}
If $v\in [0,1]$, then $t\in[1,\infty)$. Therefore the integral
becomes
\begin{equation*}
\begin{split}
2H\biggl(\frac{2}{3}\biggr)+\frac{1}{2}H\biggl(\frac{1}{6}\biggr)+2H\biggl(\frac{1}{3}\biggr)
&=4\pi\int_{1}^{\infty}\log\frac{1}{1-t+t^2}\,\d\arctan\frac{\sqrt{3}t}{2-t}
\\
&=-2\pi\sqrt{3}\int_{1}^{\infty}\frac{\log(1-t+t^2)}{1-t+t^2}\,\d t
\\
&=-\frac{4\pi^2}{3}\log3-2\pi\sqrt{3}L(\chi_{-3},2).
\end{split}
\end{equation*}
\texttt{Mathematica} evaluated the final integral after we made
the substitution $t=(1+\sqrt{3}\tan u)/2$. We can
eliminate $H(1/3)$ by appealing to~\eqref{evaluation of H(1/3)},
and the identity finally reduces to~\eqref{H(1/6) and H(2/3)
relation}.

\section{Simplification for $x=5$}
\label{sec5}

Now we will find a formula for the conductor $15$ elliptic curve.
The ultimate goal of the following discussion is to obtain
Proposition~\ref{Conductor 15 reduction theorem} and formula~\eqref{elementary integral reduction part 1} below.

When $x=5$ we can use \eqref{H main theorem}, \eqref{evaluation of H(1/3)} and
\eqref{evaluation of F(3,5) in terms of H} to write
\begin{equation*}
45L(E_{15},2)+2\pi\sqrt{3}L(\chi_{-3},2)+\frac{4\pi^2}{3}\log3
=-4\pi\int_{v\in[0,1]}\log x\,\d\arctan(\sqrt{3}y),
\end{equation*}
where
\begin{equation*}
x=\frac{1-R+v}{u}, \quad y=\frac{1+R}{1-R-2v}.
\end{equation*}
The algebraic relations between $u$, $v$, and $R$ are given by
\begin{equation*}
R^3-3v R-v^3=1-u^3, \quad
(u+v-1)^2-9u v=0.
\end{equation*}
Eliminating $u$, $v$, and $R$ with successive resultants leads to
a relation between $x$ and~$y$:
\begin{equation*}
\begin{split}
0&=(1+x+x^2)(1-15x+9x^2)+(4+20x-12x^2)y
\\ &\quad
+(6-44x-18x^2-36x^3+54x^4)y^2
+(4+60x-36x^2)y^3
\\ &\quad
+(1-9x+9x^2)(1+3x+9x^2)y^4.
\end{split}
\end{equation*}
According to \texttt{Maple} this relation defines an elliptic
curve. It is therefore possible to parameterize $x$ and $y$ by the
Weierstrass coordinates of an elliptic curve. Assisted by \texttt{Maple}'s
\emph{Weierstrassform} routine we discovered the following parametric formulas:
\begin{align}
x&=\frac{(1-t)^2(3t+t^2-\sqrt{3}\sqrt{-3+t^2+2t^3})}{(3+t^2)^2},
\label{X parametric}\\
y&=\frac{(1+t)(3-6t-t^2-2\sqrt{3}\sqrt{-3+t^2+2t^3})}{(3-t)(3+t^2)}.
\label{Y parametric}
\end{align}
Notice that if $v\in[0,1]$, then $x\in[1,0]$ and
$t\in(\infty,1]$. We have the following proposition.

\begin{proposition}\label{Conductor 15 reduction theorem}
The following formula is true:
\begin{align}\label{conductor 15 elementary integral}
&
45L(E_{15},2)+2\pi\sqrt{3}L(\chi_{-3},2)+\frac{4\pi^2}{3}\log(3)
\\ &\quad
=4\pi\int_{1}^{\infty}\log\frac{(1-t)^2(3t+t^2-\sqrt{3}\sqrt{-3+t^2+2t^3})}{(3+t^2)^2}
\nonumber
\\ &\quad\qquad\times
\d\arctan\frac{\sqrt{3}(1+t)(3-6t-t^2-2\sqrt{3}\sqrt{-3+t^2+2t^3})}{(3-t)(3+t^2)}.
\nonumber
\end{align}
\end{proposition}

Despite the fact that \eqref{conductor 15 elementary integral} is
easy to compute numerically, it is still too complicated in its
present form.  The PSLQ algorithm was instrumental in discovering
the following steps. First notice that the differential splits
into two pieces. The following identity is trivial to verify with
a computer:
\begin{equation*}
\d\arctan(\sqrt{3}y)
=2\d\arctan\frac{t}{\sqrt{3}}
+\d\arctan\frac{(3-t)(3+3t+2t^2)}{3(1+t)\sqrt{-3+t^2+2t^3}}.
\end{equation*}
Furthermore, if we introduce the real Galois conjugate of~$x$,
\begin{equation}\label{X conjugate}
\bar{x}=\frac{(1-t)^2(3t+t^2+\sqrt{3}\sqrt{-3+t^2+2t^3})}{(3+t^2)^2},
\end{equation}
then \eqref{conductor 15 elementary integral} can be broken into
four integrals.  We have
\begin{align}\label{elementary integral reduction part 1}
&
45L(E_{15},2)+2\pi\sqrt{3}L(\chi_{-3},2)+\frac{4\pi^2}{3}\log(3)
\\ &\quad
=4\pi\int_{1}^{\infty}\log(x\bar{x})\,\d\arctan\frac{t}{\sqrt{3}}
+4\pi\int_{1}^{\infty}\log\frac{x}{\bar{x}}\,\d\arctan\frac{t}{\sqrt{3}}
\nonumber
\\ &\quad\qquad
+2\pi\int_{1}^{\infty}\log(x\bar{x})\,\d\arctan\frac{(3-t)(3+3t+2t^2)}{3(1+t)\sqrt{-3+t^2+2t^3}}
\nonumber
\\ &\quad\qquad
+2\pi\int_{1}^{\infty}\log\frac{x}{\bar{x}}\,\d\arctan\frac{(3-t)(3+3t+2t^2)}{3(1+t)\sqrt{-3+t^2+2t^3}},
\nonumber
\end{align}
where $x$ and $\bar{x}$ are defined in \eqref{X parametric} and~\eqref{X conjugate}.
It is unfortunate that the integrals in \eqref{elementary integral
reduction part 1} are so complicated.  We will simplify all four
integrals in the following four lemmas. Two of them reduce to the
desired quantities almost immediately.

\begin{lemma}
\label{piece1}
The following evaluation holds:
\begin{equation}\label{D2 formula part 2}
\int_{1}^{\infty}\log(x\bar{x})\,\d\arctan\frac{t}{\sqrt{3}}
=-\sqrt{3}L(\chi_{-3},2)-\frac{2\pi}{3}\log 3.
\end{equation}
\end{lemma}

\begin{proof}
First set $t=\sqrt{3}\tan\theta$, and notice
\begin{align}
\label{X times barX}
x\bar{x}&=\frac{(1-t)^4}{(3+t^2)^2}
=\frac{16}{9}\sin^4\biggl(\theta-\frac{\pi}{6}\biggr),
\\
\label{X plus barX} x+\bar{x} &=\frac{2t(1-t)^2(3+t)}{(3+t^2)^2}
=\frac{16}{3}\sin^2\biggl(\theta-\frac{\pi}{6}\biggr)\cos\biggl(\theta-\frac{\pi}{6}\biggr)\sin\theta.
\end{align}
It follows immediately that
\begin{align*}
\int_{1}^{\infty}\log(x\bar{x})\,\d\arctan\frac{t}{\sqrt{3}}
&=\int_{\pi/6}^{\pi/2}\log\biggl(\frac{16}{9}\sin^4\biggl(\theta-\frac{\pi}{6}\biggr)\biggr)\d\theta
\\
&=4\int_{0}^{\pi/3}\log(2\sin\theta)\,\d\theta-\frac{2\pi}{3}\log 3
\\
&=-\sqrt{3}L(\chi_{-3},2)-\frac{2\pi}{3}\log 3,
\end{align*}
where the last step makes use of standard evaluations of the Clausen functions.
\end{proof}

\begin{lemma}
\label{piece2}
We have
\begin{equation}\label{m(1) easy formula}
\int_{1}^{\infty}\log\frac{x}{\bar{x}}\,\d\arctan\frac{t}{\sqrt{3}}
=-2\pi\,\m\biggl(1+X+\frac{1}{X}+Y+\frac{1}{Y}\biggr).
\end{equation}
\end{lemma}

\begin{proof}
Notice that
$\bar{x}/x>1$ and $0<x/\bar{x}<1$, whenever $t\in[1,\infty)$.
Therefore, if $t=\sqrt{3}\tan\theta$ and
$\theta\in[\pi/6,\pi/2]$, by Jensen's formula
\begin{equation*}
\begin{split}
\log\frac{x}{\bar{x}}
&=-\m\biggl(\biggl(Z-\frac{x}{\bar{x}}\biggr)\biggl(Z-\frac{\bar{x}}{x}\biggr)\biggr)
\\
&=-\m\biggl((Z+1)^2-16Z\cos^2\biggl(\theta-\frac{\pi}{6}\biggr)\sin^2\theta\biggr),
\end{split}
\end{equation*}
where we simplified the polynomial using \eqref{X times barX}
and~\eqref{X plus barX}. Also observe that if
$\theta\in[0,\pi/6]\cup[\pi/2,\pi]$, then
$|x/\bar{x}|=|\bar{x}/x|=1$.  In those cases the Mahler measure is
identically zero. Therefore, we can write
\begin{align*}
\int_{1}^{\infty}\log\frac{x}{\bar{x}}\,\d\arctan\frac{t}{\sqrt{3}}
&=-\int_{0}^{\pi}\m\biggl((Z+1)^2-16Z\cos^2\biggl(\theta-\frac{\pi}{6}\biggr)\sin^2\theta\biggr)\d\theta
\\
\intertext{(introducing the notation $T=e^{i\theta}$ and
$\zeta=e^{-\pi i/6}$)}
&=-\int_{0}^{\pi}\m\bigl((Z+1)^2+Z(T^2\zeta-T^{-2}\zeta^{-1}+i)^2\bigr)\d\theta
\\
&=-\frac12\int_{0}^{2\pi}\m\bigl((Z+1)^2+Z(T\zeta-T^{-1}\zeta^{-1}+i)^2\bigr)\d\theta
\\
&=-\pi\,\m\bigl((Z+1)^2+Z(T\zeta-T^{-1}\zeta^{-1}+i)^2\bigr)
\\
\intertext{(using the substitution
$(Z,T)\mapsto(X^2,i\zeta^{-1}Y)$ and the elementary properties of
Mahler measures)}
&=-\pi\,\m\bigl((X^2+1)^2-X^2(Y+Y^{-1}+1)^2\bigr)
\\
&=-2\pi\,\m(1+X+X^{-1}+Y+Y^{-1}).
\end{align*}
\vskip-\baselineskip
\end{proof}

\begin{lemma}
\label{piece3}
The following formula is true:
\begin{align}\label{D2 formula}
&
\int_{1}^{\infty}\log\frac{x}{\bar{x}}\,\d\arctan\frac{(3-t)(3+3t+2t^2)}{3(1+t)\sqrt{-3+t^2+2t^3}}
\\ &\quad
=2\pi\,\m\bigl(-Y^2+X(1-Y-2Y^2-Y^3+Y^4)-X^2Y^2\bigr).
\nonumber
\end{align}
\end{lemma}

\begin{proof}
The proof follows from
parameterizing the integral differently.  If we let
\begin{equation*}
u:=\sqrt{\frac{x}{\bar{x}}}
=\sqrt{\frac{3t+t^2-\sqrt{3(-3+t^2+2t^3)}}{3t+t^2+\sqrt{3(-3+t^2+2t^3)}}}
\end{equation*}
and
\begin{equation*}
v:=\frac{(3-t)(3+3t+2t^2)}{3(1+t)\sqrt{-3+t^2+2t^3}},
\end{equation*}
then
\begin{equation*}
v=\pm\frac{1+u}{1-u}\sqrt{\frac{-1+3u-u^2}{1+u+u^2}};
\end{equation*}
the plus sign is chosen for $t\in[1,3]$, and the minus sign
is chosen for $t\in[3,\infty)$. Furthermore, when $t\in[1,3]$ we
have $u\in[1,(3-\sqrt{5})/2]$, and when $t\in [3,\infty)$ we
have $u\in[(3-\sqrt{5})/2,1]$. With a little work the
integral becomes
\begin{align*}
&
\int_{1}^{\infty}\log\frac{x}{\bar{x}}\,\d\arctan\frac{(3-t)(3+3t+2t^2)}{3(1+t)\sqrt{-3+t^2+2t^3}}
\\ &\quad
=4\int_{1}^{(3-\sqrt{5})/2}\log u\,\d\arctan\biggl(\frac{1+u}{1-u}\sqrt{\frac{-1+3u-u^2}{1+u+u^2}}\biggr)
\\
\intertext{(taking $r$ for $(-1+3u-u^2)/(1+u+u^2)$)} &\quad
=4\int_{1/3}^{0}\log\frac{3-r-\sqrt{5-14r-3r^2}}{2(1+r)}\,
\d\arctan\sqrt{\frac{r(5+r)}{1-3r}}\\
&\quad
=4\int_{0}^{1/3}\log\frac{3-r+\sqrt{5-14r-3r^2}}{2(1+r)}\,
\d\arctan\sqrt{\frac{r(5+r)}{1-3r}}.
\end{align*}
We can use Jensen's formula again, to substitute a one-variable
Mahler measure for the logarithmic term:
\begin{align*}
%&
%\int_{1}^{\infty}\log\frac{x}{\bar{x}}\,\d\arctan\frac{(3-t)(3+3t+2t^2)}{3(1+t)\sqrt{-3+t^2+2t^3}}
%\\
&\quad
=4\int_{0}^{1/3}\m\biggl((1-Y)(1-Y^3)-\frac{4(1-3r)}{(1+r)^2}Y^2\biggr)
\d\arctan\sqrt{\frac{r(5+r)}{1-3r}};
\end{align*}
note that the polynomial
$$
(1-Y)(1-Y^3)-\frac{4(1-3r)}{(1+r)^2}Y^2
$$
has only one zero outside the unit circle for $r\in[0,1/3]$.
Finally, if $r(5+r)/(1-3r)=\tan^2\theta$, then $(1-3r)/(1+r)^2=\cos^2\theta$
and the integral becomes
\begin{align*}
%&
%\int_{1}^{\infty}\log\frac{x}{\bar{x}}\,\d\arctan\frac{(3-t)(3+3t+2t^2)}{3(1+t)\sqrt{-3+t^2+2t^3}}
%\\
&\quad
=4\int_{0}^{\pi/2}\m\bigl((1-Y)(1-Y^3)-4Y^2\cos^2\theta\bigr)\d\theta
\\ &\quad
=\int_{0}^{2\pi}\m\bigl((1-Y)(1-Y^3)-4Y^2\cos^2\theta\bigr)\d\theta
\\ &\quad
=2\pi\,\m\bigl((1-Y)(1-Y^3)X-Y^2(X+1)^2\bigr),
\end{align*}
which expands into~\eqref{D2 formula}.
%We will discuss the implications of this identity momentarily.
\end{proof}

\begin{lemma}
\label{piece4}
The following formula is valid:
\begin{align}\label{m(1) difficult formula}
&
\int_{1}^{\infty}\log(x\bar{x})\,\d\arctan\frac{(3-t)(3+3t+2t^2)}{3(1+t)\sqrt{-3+t^2+2t^3}}
\\ &\quad
=2\pi\log3-2\int_{0}^{1}\frac{(3+2u)\log u}{\sqrt{u(1-u)(3+u)(4+u)}}\,\d u.
\nonumber
\end{align}
\end{lemma}

\begin{proof}
Let us begin by substituting \eqref{X times barX} for $x\bar{x}$ and simplifying
the differential. We have
\begin{align*}
&
\int_{1}^{\infty}\log(x\bar{x})\,
\d\arctan\frac{(3-t)(3+3t+2t^2)}{3(1+t)\sqrt{-3+t^2+2t^3}}
\\ &\quad
=3\int_{1}^{\infty}\frac{1-4t-t^2}{(3+t^2)\sqrt{-3+t^2+2t^3}}
\,\log\frac{(t-1)^2}{t^2+3}\,\d t
\\
\intertext{(after letting $t\mapsto (t+3)/(t-1)$)}
&\quad
=3\int_{1}^{\infty}\frac{1-4t-t^2}{(3+t^2)\sqrt{-3+t^2+2t^3}}
\,\log\frac4{t^2+3}\,\d t
\\
\intertext{(averaging the last two integrals)}
&\quad
=3\int_{1}^{\infty}\frac{1-4t-t^2}{(3+t^2)\sqrt{-3+t^2+2t^3}}
\,\log\frac{2(t-1)}{t^2+3}\,\d t.
\end{align*}
If we let $u/3=2(t-1)/(t^2+3)$, then the integral splits into two
parts for $t\in [1,3]$ and $t\in[3,\infty)$.  Some work reduces the
entire expression to
\begin{align*}
%&
%\int_{1}^{\infty}\log(x\bar{x})\,
%\d\arctan\frac{(3-t)(3+3t+2t^2)}{3(1+t)\sqrt{-3+t^2+2t^3}}
%\\
&\quad
=-2\int_{0}^{1}\frac{(3+2u)\log(u/3)}{\sqrt{u(1-u)(3+u)(4+u)}}\,\d u
\\ &\quad
=2\pi\log3-2\int_{0}^{1}\frac{(3+2u)\log u}{\sqrt{u(1-u)(3+u)(4+u)}}\,\d u,
\end{align*}
where on the final step we used the formula
\begin{align*}
\int_{0}^{1}\frac{(3 + 2u)\,\d u}{\sqrt{u(1-u)(3+u)(4+u)}}
&=\int_{0}^{1}\frac{\d(u(u+3))}{\sqrt{u(3+u)(4-u(3+u))}}
\\
&=\int_{0}^{4}\frac{\d v}{\sqrt{v(4-v)}}
=\pi.
\end{align*}
\vskip-\baselineskip
\end{proof}

While formulas \eqref{D2 formula part 2} and \eqref{m(1) easy
formula} have been reduced as far as possible, formulas \eqref{D2
formula} and \eqref{m(1) difficult formula} require slightly more
attention.

The right-hand side of formula \eqref{D2 formula} is
\emph{extremely} surprising. The polynomial inside the Mahler
measure is a knot invariant; namely,
$$
A(X,Y):=-Y^2+X(1-Y-2Y^2-Y^3+Y^4)-X^2Y^2
$$
is the A-polynomial of the figure eight knot, denoted $4_1$ by Rolfson \cite{Ro}. Boyd discussed
this particular polynomial in great detail \cite{BoKnot}. Its normalized Mahler
measure, $\pi\,\m(A)$, equals the volume of the hyperbolic
manifold obtained from the complement of $4_1$ in the $3$-sphere.
These volumes are well defined, and can be calculated in terms of
values of the Bloch--Wigner dilogarithm~\cite{BoRV}. The end result of that
analysis is the following identity:
\begin{equation}\label{knot dilog}
\pi\,\m\bigl(-Y^2+X(1-Y-2Y^2-Y^3+Y^4)-X^2Y^2\bigr)
=\frac{3\sqrt{3}}{2}L(\chi_{-3},2).
\end{equation}
Boyd has also informed us that Rodriguez-Villegas gave the first
proof of this result.  Although
$A(X,Y)=0$ defines a conductor 15 elliptic curve,
we are at a loss to explain this surprising
appearance of knot theory.  We will speculate that it must be deeply
connected to some type of underlying geometry associated with the
signature~3 modular equations.

Formula \eqref{m(1) difficult formula} is an analogue of
the integrals for the conductor $20$ and $24$ elliptic curves
\cite{RZ}, and we will reduce it to a Mahler measure
using a similar approach. For this, we introduce the integral
\begin{equation}
I(y):=-\frac2\pi\int_0^1\frac{(y-1+2u)\log u}{\sqrt{u(1-u)(y-1+u)(y+u)}}\,\d u.
\label{I(y)}
\end{equation}

\begin{proposition}
\label{prop-I(y)}
For $y\ge1$, the following evaluation is valid:
\begin{align}
I(y)
&=4\log2-\frac1{8y^2}\,{}_4F_3\biggl(\begin{matrix} \frac32, \, \frac32, \, 1, \, 1 \\
2, \, 2, \, 2 \end{matrix}\biggm| \frac1{y^2} \biggr)
-\frac1{y}\,{}_3F_2\biggl(\begin{matrix} \frac12, \, \frac12, \, \frac12 \\
1, \, \frac32 \end{matrix}\biggm| \frac1{y^2} \biggr)
\label{I(y)-hyper}
\\
&=m(4y)-m\biggl(\frac 4y\biggr)-\log\frac y4,
\label{I(y)-mahler}
\end{align}
where
$m(\alpha)=\m(\alpha+X+X^{-1}+Y+Y^{-1})$.
\end{proposition}

\begin{proof}
First of all note that the integral $I(y)$ can be written as
$$
I(y)=-\frac2\pi\int_0^1\log u\cdot\frac{\partial w}{\partial u}\,\d u,
$$
where
$$
w(u)=w(u;y):=\arcsin\frac{2u(y-1+u)-y}{y}
$$
and, for $y\ge1$, the argument
$$
\frac{2u(y-1+u)-y}{y}
$$
monotonically changes from $-1$ to $1$ when $u\in[0,1]$.
Since
$$
\frac{\partial w}{\partial y}=\frac1y\,\frac{u(1-u)}{\sqrt{u(1-u)(y-1+u)(y+u)}},
$$
the integration by parts for $y>1$ results in
\begin{align*}
\frac{\d I}{\d y}
&=-\frac2\pi\int_0^1\log u\,\d\biggl(\frac{\partial w}{\partial y}\biggr)
\\
&=-\frac2\pi\,\log u\cdot\frac{\partial w}{\partial y}\bigg|_{u=0}^{u=1}
+\frac2\pi\int_0^1\frac{\partial w}{\partial y}\,\frac{\d u}u
\\
&=\frac1{\pi y}\int_0^1\frac{2(1-u)\,\d u}{\sqrt{u(1-u)(y-1+u)(y+u)}}.
\end{align*}
Consider also the related integral
$$
\frac{y+1}{\pi}\int_0^1\frac{\d u}{\sqrt{u(1-u)(y-1+u)(y+u)}}
=\frac{2}{\pi}K\biggl(\frac{2\sqrt y}{y+1}\biggr),
$$
where
$$
K(z)=\frac\pi2\,{}_2F_1\biggl(\begin{matrix} \frac12, \, \frac12 \\ 1 \end{matrix}\biggm|
z^2 \biggr), \qquad |z|\le1,
$$
is the complete elliptic integral of the first kind.
On using the Gauss quadratic transformation
$$
K(z)=\frac1{1+z}K\biggl(\frac{2\sqrt z}{1+z}\biggr)
$$
with the choice $z=1/y$ (hence $0<z<1$ for $y>1$), we obtain
\begin{align*}
\frac{y+1}{\pi}\int_0^1\frac{\d u}{\sqrt{u(1-u)(y-1+u)(y+u)}}
&=\frac{2}{\pi}\cdot\frac{y+1}yK\biggl(\frac1y\biggr)
\\
&=\frac{y+1}y\,{}_2F_1\biggl(\begin{matrix} \frac12, \, \frac12 \\ 1 \end{matrix}\biggm|
\frac1{y^2} \biggr).
\end{align*}
Therefore,
\begin{align}
\label{simpleI(y)}
\frac{y+1}y\,{}_2F_1\biggl(\begin{matrix} \frac12, \, \frac12 \\ 1 \end{matrix}\biggm|
\frac1{y^2} \biggr)
-y\frac{\d I}{\d y}
&=\frac1\pi\int_0^1\frac{(y-1+2u)\,\d u}{\sqrt{u(1-u)(y-1+u)(y+u)}}
\\
&=\frac1\pi\int_0^1\frac{\partial w}{\partial u}\,\d u
=\frac1\pi\bigl(w(1;y)-w(0;y)\bigr)=1,
\nonumber
\end{align}
and so we have
\begin{align*}
y\frac{\d I}{\d y}
&=-1+\frac{y+1}y\,{}_2F_1\biggl(\begin{matrix} \frac12, \, \frac12 \\ 1 \end{matrix}\biggm|
\frac1{y^2} \biggr)
=\sum_{n=1}^\infty\frac{(\frac12)_n^2}{n!^2}\,\frac1{y^{2n}}
+\sum_{n=0}^\infty\frac{(\frac12)_n^2}{n!^2}\,\frac1{y^{2n+1}}
\\
&=-y\frac{\d}{\d y}\biggl(\frac1{8y^2}\,{}_4F_3\biggl(\begin{matrix} \frac32, \, \frac32, \, 1, \, 1 \\
2, \, 2, \, 2 \end{matrix}\biggm| \frac1{y^2} \biggr)
+\frac1{y}\,{}_3F_2\biggl(\begin{matrix} \frac12, \, \frac12, \, \frac12 \\
1, \, \frac32 \end{matrix}\biggm| \frac1{y^2} \biggr) \biggr).
\end{align*}
The integration gives us
$$
I(y)
=C-\frac1{8y^2}\,{}_4F_3\biggl(\begin{matrix} \frac32, \, \frac32, \, 1, \, 1 \\
2, \, 2, \, 2 \end{matrix}\biggm| \frac1{y^2} \biggr)
-\frac1{y}\,{}_3F_2\biggl(\begin{matrix} \frac12, \, \frac12, \, \frac12 \\
1, \, \frac32 \end{matrix}\biggm| \frac1{y^2} \biggr).
$$
To compute the constant of integration we use definition~\eqref{I(y)} of the integral $I(y)$:
$$
C=\lim_{y\to+\infty}I(y)
=-\frac2\pi\int_0^1\frac{\log t}{\sqrt{t(1-t)}}\,\d t
=4\log2.
$$
Although we have done the computation for $y>1$, the resulting formula~\eqref{I(y)-hyper}
is valid for $y\ge1$ because of continuity at $y=1$. The two hypergeometric
series can be further reduced to the Mahler measures by using the
formulas~\cite{KO},~\cite{RV},~\cite{Rgsubmit}
\begin{equation*}
m(\alpha)=\log\alpha
-\frac{2}{\alpha^2}\,{}_4F_3\biggl(\begin{matrix} \frac32, \, \frac32, \, 1, \, 1 \\
2, \, 2, \, 2 \end{matrix}\biggm| \frac{16}{\alpha^2} \biggr)
%\label{hyper-m1}
\end{equation*}
and
\begin{equation*}
m(\alpha)=\frac\alpha4\,{}_3F_2\biggl(\begin{matrix} \frac12, \, \frac12, \, \frac12 \\
1, \, \frac{3}{2} \end{matrix}\biggm| \frac{\alpha^2}{16} \biggr)
%\label{hyper-m2}
\end{equation*}
for $\alpha\ge4$ and $0<\alpha\le4$, respectively. This proves formula~\eqref{I(y)-mahler}.
\end{proof}

On invoking the computation in~\eqref{simpleI(y)} we can also
state formula~\eqref{I(y)-mahler} in the form
\begin{equation*}
\frac2\pi\int_0^1\frac{(y-1+2u)\log\dfrac{\sqrt y}{2u}}{\sqrt{u(1-u)(y-1+u)(y+u)}}\,\d u
=m(4y)-m\biggl(\frac 4y\biggr).
%\label{I(y)-mahler2}
\end{equation*}
When $y=4$ we obtain
\begin{align}
\label{I(4)}
-\frac2\pi\int_0^1\frac{(3+2u)\log u}{\sqrt{u(1-u)(3+u)(4+u)}}\,\d u
&=m(16)-m(1)=10m(1)
\\
&=10\m(1+X+X^{-1}+Y+Y^{-1}),
\nonumber
\end{align}
with the linear relation between $m(1)$ and $m(16)$ recently obtained
by Lal\'\i n~\cite{La} (see also \cite{GuR} for an elementary proof).

Combining \eqref{elementary integral reduction part 1},
\eqref{D2 formula part 2}, \eqref{m(1) easy formula},
\eqref{D2 formula}, \eqref{m(1) difficult formula},
\eqref{knot dilog}, and \eqref{I(4)}, we finally arrive at

\begin{maintheorem}
The following relation holds true:
$$
L(E_{15},2)=\frac{4\pi^2}{15}\m(1+X+X^{-1}+Y+Y^{-1}).
$$
\end{maintheorem}

\section{Conclusion}

This work has raised a number of questions which are worth
mentioning.  The first is whether or not it is possible to say
something about the $L$-functions of conductor $33$ elliptic
curves. Equations \eqref{F(1,11) and F(3,11) in terms of H} and
\eqref{H main theorem} can be used to produce a `coupled'
identity, relating $L(E_{11},2)$ and $L(E_{33},2)$ to an
elementary integral. Unfortunately, the integral presents an
enormous obstacle.  The polynomial relating $v$ and $R$ (obtained
from eliminating $u$ in \eqref{H main theorem}) has genus $3$.
\texttt{Maple} failed to find parametric formulas, and our
analysis stalled.  As a final complication, Boyd's paper does not
mention any identities involving conductor $33$ $L$-series
\cite{Bo1}.  It seems that an additional method of evaluating
\eqref{H main theorem} is needed.

It is also worth understanding why our method produces so many
`coupled' identities.  Perhaps one explanation for the conductor
$11$--$33$ pair, is that they both arise from integrating modular
forms on the same congruence subgroup, $\Gamma_{0}(33)$. We have
also produced a massively complicated formula for the conductor
$24$--$48$ pair, which may have a similar justification. The final
puzzling aspect of this work is that our proof of the conductor $15$
case required the non-trivial evaluation~\eqref{knot dilog}
of the Mahler measure of a knot polynomial. The fact that the conductor~$15$
$L$-series couples to a Dirichlet $L$-series, was probably fortunate
for our calculations.

\begin{acknowledgements}
The authors would like to thank David Boyd and Nathan Dunfield
for their valuable comments and details on identity~\eqref{knot dilog}.
We kindly acknowledge Bruce Berndt, David Boyd, Christopher Deninger,
and Michael Somos for their useful comments and encouragement.
%Special thanks are due to ... for helpful assistance in numerics and experimentation as well as for many fruitful suggestions.
\end{acknowledgements}


\begin{thebibliography}{99}

\bibitem{Be5}
\textsc{B.\,C.~Berndt},
\emph{Ramanujan's Notebooks, Part V}
(Springer-Verlag, New York, 1998).

\bibitem{BE}
\textsc{M.\,J.~Bertin},
Mesure de Mahler d'une famille de polyn\^omes,
\emph{J. Reine Angew. Math.} \textbf{569} (2004), 175--188.

\bibitem{Borwein}
\textsc{J.\,M.~Borwein} and \textsc{P.\,B.~Borwein},
A cubic counterpart of Jacobi's identity and the AGM,
\emph{Trans. Amer. Math. Soc.} \textbf{323} (1991), 691--701.

\bibitem{BoG}
\textsc{J.\,M.~Borwein}, \textsc{P.\,B.~Borwein}, and \textsc{F.~Garvan},
Some cubic modular identities of Ramanujan,
\emph{Trans. Amer. Math. Soc.} \textbf{343} (1994), 35--47.

\bibitem{Bo1}
\textsc{D.\,W.~Boyd},
Mahler's measure and special values of $L$-functions,
\emph{Experiment. Math.} \textbf{7} (1998), 37--82.

\bibitem{BoKnot}
\textsc{D.\,W.~Boyd},
Mahler's measure and invariants of hyperbolic manifolds,
in: \emph{Number theory for the millennium},
M.\,A.~Bennett et al., eds.
(A K Peters, Boston, 2002), 127--143.

\bibitem{BoRV}
\textsc{D.\,W.~Boyd} and \textsc{F.~Rodriguez-Villegas},
Mahler's measure and the dilogarithm (II),
with an appendix by N.\,M.~Dunfield,
preprint \texttt{arXiv:\,math/0308041 [math.NT]} (2003).

\bibitem{Br}
\textsc{F.~Brunault},
Version explicite du th\'eor\`eme de Beilinson pour la courbe modulaire $X_1(N)$,
\emph{C. R. Math. Acad. Sci. Paris} \textbf{343} (2006), no.~8, 505--510.

\bibitem{De}
\textsc{C.~Deninger},
Deligne periods of mixed motives, $K$-theory and the entropy of certain $\mathbb Z^n$-actions,
\emph{J. Amer. Math. Soc.} \textbf{10} (1997), no.~2, 259--281.

\bibitem{GuR}
\textsc{J.~Guillera} and \textsc{M.~Rogers},
Mahler measure and the WZ algorithm,
preprint \texttt{arXiv:\,1006.1654 [math.NT]} (2010).

\bibitem{GK}
\textsc{G.~K\"ohler},
\emph{Eta products and theta series identities}
(Springer-Verlag, Heidelberg, 2011).

\bibitem{KZ}
\textsc{M.~Kontsevich} and \textsc{D.~Zagier},
Periods,
in: \emph{Mathematics unlimited\,---\,2001 and beyond}
(Springer, Berlin, 2001), 771--808.

\bibitem{KO}
\textsc{N.~Kurokawa} and \textsc{H.~Ochiai},
Mahler measures via crystalization,
\emph{Comment. Math. Univ. St. Pauli} \textbf{54} (2005), 121--137.

\bibitem{La}
\textsc{M.\,N.~Lal\'\i n},
On a conjecture by Boyd,
\emph{Int. J. Number Theory} \textbf{6} (2010), no.~3, 705--711.

\bibitem{Ono}
\textsc{Y.~Martin} and \textsc{K.~Ono},
Eta-quotients and elliptic curves,
\emph{Proc. Amer. Math Soc.} \textbf{125} (1997), no.~11, 3169--3176.

\bibitem{Me}
\textsc{A.~Mellit},
Elliptic dilogarithms and parallel lines,
preprint (2009).

\bibitem{Pat}
\textsc{D.~Pathakjee}, \textsc{Z.~Rosnbrick}, and \textsc{E.~Yoong},
Elliptic curves, eta quotients and hypergeometric functions,
preprint (2010).

\bibitem{RV}
\textsc{F.~Rodriguez-Villegas},
Modular Mahler measures I,
in: \emph{Topics in number theory} (University Park, PA, 1997),
Math. Appl. \textbf{467} (Kluwer Acad. Publ., Dordrecht, 1999), 17--48.

\bibitem{Ro}
\textsc{D.~Rolfsen},
\emph{Knots and links} (Publish or Perish, Berkeley, 1976).

\bibitem{Rgsubmit}
\textsc{M.~Rogers},
Hypergeometric formulas for lattice sums and Mahler measures,
\emph{Intern. Math. Res. Not.} (to appear),
preprint \texttt{arXiv:\,0806.3590 [math.NT]} (2008).

\bibitem{RZ}
\textsc{M.~Rogers} and \textsc{W.~Zudilin},
{}From $L$-series of elliptic curves to Mahler measures,
preprint \texttt{arXiv:\,1012.3036 [math.NT]} (2010).

\bibitem{Sm}
\textsc{C.\,J.~Smyth},
On measures of polynomials in several variables,
\emph{Bull. Austral. Math. Soc.} \textbf{23} (1981), no.~1, 49--63.

\bibitem{so2}
\textsc{M.~Somos},
Dedekind eta function product identities,
available at \url{http://eta.math.georgetown.edu/}.

\end{thebibliography}
\end{document}